\documentclass{amsart}
\usepackage[margin=2.5cm]{geometry}
\usepackage{amsmath,amsfonts,amssymb,amsthm}
\usepackage{graphics}
\usepackage{epsfig, esint}
\usepackage{graphics,color}
\usepackage{paralist}

   \definecolor{labelkey}{gray}{.8}
   \definecolor{refkey}{gray}{.8}

\providecommand{\R}{\mathbb{R}}

\providecommand{\osc}{\operatorname{osc}}

\newcommand{\e}{\varepsilon}

\newcommand{\step}[1]{\medskip\noindent\textbf{Step #1. }}

\newcommand{\ignore}[1]{}

\newtheorem{definition}{Definition}
\newtheorem{proposition}{Proposition}
\newtheorem{theorem}{Theorem}
\newtheorem{remark}{Remark}
\newtheorem{lemma}{Lemma}
\newtheorem{corollary}{Corollary}
\newtheorem{assumption}[theorem]{Assumption}

\author[P.~Bella]{Peter Bella}
\address{Mathematisches Institut 
 Universit\"at Leipzig\\
 Leipzig, 04103 Germany.}
\email{bella@math.uni-leipzig.de}
\author[M. Sch\"affner]{Mathias Sch\"affner}
\address{Mathematisches Institut 
 Universit\"at Leipzig\\
 Leipzig, 04103 Germany.}
\email{schaeffner@math.uni-leipzig.de}

\title[Regularity for non-uniformly elliptic equations]{Local boundedness and Harnack inequality for solutions of linear non-uniformly elliptic equations} 

\begin{document}
\maketitle


\begin{abstract}

We study local regularity properties for solutions of linear, non-uniformly elliptic equations. Assuming certain integrability conditions on the coefficient field, we prove local boundedness and Harnack inequality. The assumed integrability assumptions are essentially sharp and improve  upon classical results by Trudinger [ARMA 1971]. We then apply the deterministic regularity results to the corrector equation in stochastic homogenization and establish sublinearity of the corrector.

\medskip

\end{abstract}

\section{Introduction and main results}

We consider linear, second order, scalar elliptic equations in divergence form,
\begin{equation}\label{eq}
 -\nabla \cdot a \nabla u = 0,
\end{equation}
where $a:\Omega\to\R^{d\times d}$ is a measurable matrix field on a domain $\Omega\subset\R^d$, $d\geq2$. In order to measure ellipticity of $a$, we introduce
\begin{equation}\label{def:lmu}
 \lambda(x):=\inf_{\xi\in\R^d} \frac{\xi\cdot a(x)\xi}{|\xi|^2},\qquad \mu(x):=\sup_{\xi\in\R^d}\frac{|a(x)\xi|^2}{\xi\cdot a(x)\xi}
\end{equation}
and suppose that $\lambda$ and $\mu$ are measurable non-negative functions. If $\lambda^{-1}$ and $\mu$ are essentially bounded (i.e.\ $a$ is uniformly elliptic), the seminal contributions of DeGiorgi \cite{DG57} and Nash \cite{Nash58} ensure that weak solutions of \eqref{eq} are H\"older continuous. Moreover, Moser \cite{Moser60,Moser61} showed that weak solutions of \eqref{eq} satisfy the Harnack inequality which then implies H\"older continuity. Here, we are interested in situations beyond the uniform ellipticity.

\smallskip

In \cite{T71} Trudinger considered non-uniformly elliptic equations of the type \eqref{eq}. Instead of essential boundedness, he assumed that  $\lambda^{-1}\in L^q(\Omega)$ and $\mu\in L^p(\Omega)$ with $\frac1p+\frac1q<\frac2d$ and proved that weak solutions to \eqref{eq} are locally bounded and satisfy the Harnack inequality. In this paper, we prove both results under the less restrictive and essentially optimal assumption $\frac1p+\frac1q<\frac2{d-1}$. More precisely, we establish the following
\begin{theorem}\label{T:1}
Fix $d\geq2$, a domain $\Omega\subset\R^d$ and $p,q\in(1,\infty]$ satisfying
\begin{equation}\label{eq:pq}
 \frac 1p + \frac 1q < \frac{2}{d-1}.
\end{equation}
Let $a:\Omega\to\R^{d\times d}$ be such that $\lambda$ and $\mu$ given in \eqref{def:lmu} are non-negative and satisfy $\frac1\lambda\in L^q(\Omega)$, $\mu\in L^p(\Omega)$. Then any weak solution $u$ of \eqref{eq} in $\Omega$ satisfies:
\begin{enumerate}[(i)]
\item (Local boundedness) For every $\gamma>0$ there exists $c=c(\gamma,d,p,q)\in[1,\infty)$ such that for any ball $B_R\subset\Omega$, $R>0$, it holds
\begin{equation}\label{est:C:bound}
 \|u\|_{L^\infty(B_{\frac{R}2})} \leq c\Lambda(B_R)^{\frac{p'}{\gamma}(1+\frac1\delta)}\left(\fint_{B_R} |u|^\gamma\right)^\frac1\gamma,
\end{equation}
where $\delta:=\min\{\frac{1}{d-1}-\frac1{2p},\frac12\}-\frac1{2q}>0$, $p':=\frac{p}{p-1}$ and for every measurable set $S\subset\Omega$
\begin{equation}\label{def:lambda}
 \Lambda(S):=\left(\fint_S \mu^p\right)^\frac1p\left(\fint_S \lambda^{-q}\right)^\frac1q.
\end{equation}
%
%
%
%
\item (Harnack inequality) If $u$ is non-negative in the ball $B_R\subset\Omega$, then 
\begin{equation}\label{est:harnack:strong}
 \sup_{B_{\frac{R}2}} u\leq c\inf_{B_{\frac{R}2} } u,
\end{equation}
where $c=c(d,p,q,\Lambda(B_R))\in[1,\infty)$.

\end{enumerate}
\end{theorem}

\begin{remark}
As mentioned above, the conclusions of Theorem~\ref{T:1} are proven in the classical paper of Trudinger~\cite{T71} under the more restrictive integrability condition $\frac1\lambda\in L^q(\Omega),\mu\in L^p(\Omega)$ with
\begin{equation}\label{eq:pqlooser}
 p,q\in(1,\infty],\qquad \frac1p+\frac1q<\frac2d,
\end{equation} 
see also the paper by Murthy and Stampacchia \cite{MS68} for related results. To the best of our knowledge Theorem~\ref{T:1} contains the first improvements with respect to global integrability of $\frac1\lambda$ and $\mu$, compared to the corresponding results in \cite{MS68,T71} (see \cite{CMM18} for a recent generalization of the findings in \cite{MS68,T71} to non-linear non-uniformly elliptic equations under assumptions that match \eqref{eq:pqlooser} in the linear case).  Assumption \eqref{eq:pq} is essentially sharp in order to establish local boundedness (and thus also the validity of Harnack inequality) for weak solutions of \eqref{eq}. Indeed, in view of a counterexample by Franchi, Serapioni and Serra Cassano \cite{FSS98} the conclusion of Theorem~\ref{T:1} is false if condition \eqref{eq:pq} is replaced by $\frac1p+\frac1q<\frac2{d-1}+\e$ for any $\e>0$, see Remark~\ref{rem:bound} below. However, we emphasize that under additional local assumption (e.g.\ that $\lambda,\mu$ are in the Muckenhoupt class $A_2$) stronger results are available under weaker global integrability assumptions, see e.g.\ \cite{FKS82,CW86}.
\end{remark}

\begin{remark}\label{rem:bound2d}
Note that if $\frac{1}{d-1}-\frac1{2p}-\frac1{2q}$ tends to zero from above, the prefactor on the right-hand side in \eqref{est:C:bound} blows up and we do not know if weak solutions of \eqref{eq} are locally bounded in the borderline situation $\frac1p+\frac1q=\frac2{d-1}$ in general. However, in the special case of two dimensions we are able to show local boundedness of weak solutions under the minimal assumption $p=q=1$ (and thus $\frac1p+\frac1q=2=\frac2{d-1}$), see Proposition~\ref{P:bound:2d}.
\end{remark}

As an application of Theorem~\ref{T:1} we consider the corrector equation in stochastic homogenization. Currently,  homogenization and large scale regularity for equations with random and degenerate coefficients is an active field of research, see e.g.\ \cite{ADS15,AN17,AD18,AS14,BFO18,CD16,DNS,FK97,FHS17,NSS17}. Recently, sublinearity (in $L^\infty$) of the corrector in stochastic homogenization was proven in \cite{CD16} (see also \cite{FK97}) under certain moment conditions which are comparable to \eqref{eq:pqlooser} (see also \cite{ADS15,DNS} for related results in the discrete setting). In \cite{ADS15,CD16,DNS,FK97}, the $L^\infty$-sublinearity of the corrector is the key ingredient to prove quenched invariance principles for random walks \cite{ADS15,DNS} or diffusion \cite{CD16,FK97} in a random environment with degenerate and/or unbounded coefficients.  In this paper, we establish $L^\infty$-sublinearity of the corrector under relaxed moment conditions, see Proposition~\ref{P:sublinear}.

\bigskip

The paper is organised as follows: In Section~\ref{sec:cutoff}, we present a technical lemma which implies an improved version of Caccioppoli inequality. This lemma plays a prominent role in the proof of Theorem~\ref{T:1} and is the main source for the improvement compared to the previous results in \cite{MS68,T71,T73}. In Section~\ref{sec:bound}, we make precise the notion of weak solution and prove part (i) of Theorem~\ref{T:1} and local boundedness for weak subsolution of \eqref{eq}. Section~\ref{sec:bound} contains an improvement of part (i) of Theorem~\ref{T:1} valid only in two dimensions, see Proposition~\ref{P:bound:2d}. In Section~\ref{sec:harnack}, we establish part (ii) of Theorem~\ref{T:1} as a consequence of a weak Harnack inequality for non-negative weak supersolutions of \eqref{eq} and the local boundedness. Moreover, we list in Section~\ref{sec:harnack} several direct consequences of the Harnack inequality. In the final Section~\ref{sec:prob}, we apply Theorem~\ref{T:1} to the corrector equation of stochastic homogenization and prove $L^\infty$-sublinearity of the corrector. 

\section{An auxiliary lemma}\label{sec:cutoff}

In this section, we provide a key estimate, formulated in Lemma~\ref{L:optimcutoff} below, that is central in our proof of Theorem~\ref{T:1}. 

\begin{lemma}\label{L:optimcutoff}
Fix $d\geq2$ and $p\geq1$ satisfying $p>\frac{d-1}2$ if $d\geq3$. Suppose $0<\rho<\sigma<\infty$ and $v\in W^{1,p_*}(B_\sigma)$ with $\frac{1}{p_*}=\min\{\frac12-\frac1{2p}+\frac{1}{d-1},1\}$.  Consider
\begin{equation*}
 J(\rho,\sigma,v):=\inf\left\{\int_{B_\sigma}\mu |v|^2|\nabla \eta|^2\,dx \;|\;\eta\in C_0^1(B_\sigma),\,\eta\geq0,\,\eta=1\mbox{ in $B_\rho$}\right\}.
\end{equation*}
Then there exists $c=c(d,p)\in[1,\infty)$ such that
\begin{equation}\label{est:cutoff}
J(\rho,\sigma,v)\leq c(\sigma-\rho)^{-\frac{2d}{d-1}}\|\mu\|_{L^p(B_\sigma\setminus B_\rho)}\left(\|\nabla v\|_{L^{p_*}(B_\sigma\setminus B_\rho)}^2+\rho^{-2}\|v\|_{L^{p_*}(B_\sigma\setminus B_\rho)}^2\right).
\end{equation}
\end{lemma}

\begin{proof}[Proof of Lemma~\ref{L:optimcutoff}]
\step 1 We claim 
\begin{equation}\label{1dmin}
 J(\rho,\sigma,v)\leq  (\sigma-\rho)^{-(1+\frac1\gamma)} \left(\int_{\rho}^\sigma \left(\int_{S_r} \mu |v|^2\right)^\gamma\,dr\right)^\frac1\gamma\qquad\mbox{for every $\gamma>0$}.
\end{equation}
Estimate \eqref{1dmin} follows directly by minimizing among radial symmetric cut-off functions. Indeed, we obviously have for every $\e\geq0$
\begin{equation}
 J(\rho,\sigma,v)\leq \inf\left\{\int_{\rho}^\sigma \eta'(r)^2\left(\int_{S_r}\mu |v|^2+\e\right)\,dr \;|\;\eta\in C^1(\rho,\sigma),\,\eta(\rho)=1,\,\eta(\sigma)=0\right\}=:J_{{\rm 1d},\e}.
\end{equation}
For $\e>0$, the one-dimensional minimization problem $J_{{\rm 1d},\e}$ can be solved explicitly and we obtain
\begin{equation*}
J_{{\rm 1d},\e}=\left(\int_{\rho}^\sigma \left(\int_{S_r}\mu |v|^2+\e\right)^{-1}\,dr\right)^{-1}.
\end{equation*}
%
Then H\"older inequality $\sigma-\rho=\int_\rho^\sigma \frac{f}{f}\leq\left(\int_\rho^\sigma f^s\right)^\frac1s\left(\int_\rho^\sigma \frac1{f^{s'}}\right)^\frac1{s'}$ with $s'=\frac{s}{s-1}$ and $f(r):=\left(\int_{S_r}\mu |v|^2+\e\right)^\frac{1}{s'}$ yield for any $s>1$
\begin{equation*}
J_{{\rm 1d},\e}\leq (\sigma-\rho)^{-s'}\left(\int_{\rho}^\sigma \left(\int_{S_r}\mu |v|^2+\e\right)^\frac{s}{s'}\,dr\right)^{\frac{s'}s}.
\end{equation*}
Sending $\e$ to zero, we obtain claim \eqref{1dmin} with $\gamma=s-1>0$.

\step 2 Let us first assume $d\geq3$. Note that $p>\frac{2}{d-1}$ implies $p_*\in[1,2)$. We estimate the right-hand side of \eqref{1dmin} with help of the H\"older and Sobolev inequality of the type
\begin{equation}\label{ineq:sobd3}
\forall s\in[1,d-1)\;\exists c=c(d,s):\quad \|\varphi\|_{L^{s*}(S_1)}\leq c\|\varphi\|_{W^{1,s}(S_1)}\quad\mbox{where $s^*=\frac{(d-1)s}{d-1-s}$},
\end{equation}
and $S_1=\partial B_1$. More precisely, there exists $c=c(p,d)\in[1,\infty)$
\begin{align}\label{est:J2}
J(\rho,\sigma,v)\leq&  \frac{1}{(\sigma-\rho)^{1+\frac1\gamma}} \left(\int_{\rho}^\sigma  \left(\int_{S_r} \mu^p\right)^{\frac\gamma{p}}\left(\int_{S_r} |v|^{\frac{2p}{p-1}}\right)^{\frac{(p-1)\gamma}{p}}\,dr\right)^\frac1\gamma\notag\\
\leq& \frac{c}{(\sigma-\rho)^{1+\frac1\gamma}} \left(\int_{\rho}^\sigma  \left(\int_{S_r} \mu^p\right)^{\frac\gamma{p}}\left(\left(\int_{S_r} |\nabla v|^{p_*}\right)^{\frac{2\gamma}{p_*}}+r^{-2\gamma}\left(\int_{S_r} |v|^{p_*}\right)^{\frac{2\gamma}{p_*}}\right)\,dr\right)^\frac1\gamma,
\end{align}
where $\frac{p-1}{2p}=\frac1{p_*}-\frac1{d-1}$.  The choice $\gamma=\frac{d-1}{d+1}$ yields $\frac{\gamma}{p}+\frac{2\gamma}{{p_*}}=1$ so that we can apply H\"older inequality
\begin{align*}
J(\rho,\sigma,v)\leq \frac{c}{(\sigma-\rho)^{\frac{2d}{d-1}}} \left(\int_{B_\sigma\setminus B_\rho} \mu^p\right)^\frac{1}p \left(\left(\int_{B_\sigma\setminus B_\rho}|\nabla v|^{p_*}\right)^\frac{2}{p_*}+\frac1{\rho^2}\left(\int_{B_\sigma\setminus B_\rho}|v|^{p_*}\right)^\frac{2}{p_*}\right)
\end{align*}
which is the desired estimate.

\smallskip

Finally, we suppose $d=2$. In this case we have $p_*=1$. Instead of \eqref{ineq:sobd3}, we use one-dimensional Sobolev inequality $\|\varphi\|_{L^{\infty}(S_1)}\leq c\|\varphi\|_{W^{1,1}(S_1)}$ to obtain the estimate \eqref{est:cutoff} as above (but now also in the borderline case $p=1$). 

\end{proof}

\section{Local boundedness  proof of part (i) of Theorem~\ref{T:1}}\label{sec:bound}

In this section we prove part (i) of Theorem~\ref{T:1} as a consequence of a local boundedness result for weak subsolutions of \eqref{eq}. Before we state the result, we first define the notion of weak solution to \eqref{eq} that we consider here.
\begin{definition}\label{def:solution}
Fix a domain $\Omega\subset\R^d$ and a coefficient field $a:\Omega\to\R^{d\times d}$ such that $\lambda,\mu\geq0$ given in \eqref{def:lmu} satisfy $\frac1\lambda,\mu\in L^1(\Omega)$. The spaces $H_0^1(\Omega,a)$ and $H^1(\Omega,a)$ are respectively defined as the completion of $C_0^1(\Omega)$ and $C^1(\Omega)$ with respect to the norm $\|\cdot\|_{H^1(\Omega,a)}:=\left(\mathcal A_1(\cdot,\cdot)\right)^\frac12$, where
\begin{align}\label{def:A1}
\qquad\mathcal A_1(u,v):=& \mathcal A(u,v)+\int_\Omega \mu u  v,\qquad\mbox{with}\quad\mathcal A(u,v):=\int_\Omega a\nabla u \cdot \nabla v.
\end{align}
Moreover, we denote by $H_{\rm loc}^1(\Omega,a)$ the family of functions $u$ satisfying $u\in H^1(\Omega',a)$ for every bounded open set $\Omega'\Subset\Omega$.

We call $u$ a weak solution (subsolution, supersolution) of \eqref{eq} in $\Omega$ if and only if $u\in H^1(\Omega,a)$ and
\begin{equation}\label{def:harmonic}
 \forall \phi\in H_0^1(\Omega,a),\, \phi\geq0:\qquad \mathcal A(u,\phi)=0\quad (\leq0,\geq0).
\end{equation}
Moreover, we call $u$ a local weak solution of \eqref{eq} in $\Omega$ if and only if $u$ is a weak solution of \eqref{eq} in $\Omega'$ for every bounded open set $\Omega'\Subset\Omega$. Throughout the paper, we call a solution (subsolution, supersolution) of \eqref{eq} in $\Omega$ $a$-harmonic ($a$-subharmonic, $a$-superharmonic) in $\Omega$. 
\end{definition}

For general properties of the spaces $H^1(\Omega,a)$ and $H_0^1(\Omega,a)$, we refer to \cite{T71,T73}. We only recall here the chain rule
\begin{remark}
 Let $g:\R\to\R$ be uniformly Lipschitz-continuous with $g(0)=0$ and consider the composition $F:=g(u)$. Then, $u\in H_0^1(\Omega,a)$ (or $\in H^1(\Omega,a)$) implies $F \in H_0^1(\Omega,a)$ (or $\in H^1(\Omega,a)$), and it holds $\nabla F=g'(u)\nabla u$ a.e.\ (see e.g.\ \cite[Lemma 1.3]{T73}). In particular,  if $u$ satisfies $u\in H^1(\Omega,a)$ (or $\in H^1(\Omega,a)$) then also the truncations
\begin{equation}\label{def:u+-}
 u_+:=\max\{u,0\};\quad u_{-}:=-\min\{u,0\}
\end{equation}
satisfy $u_+,u_-\in H^1(\Omega,a)$ (or $\in H^1(\Omega,a)$).
\end{remark}
Now we come to the local boundedness from above for weak subsolutions of \eqref{eq}.
\begin{theorem}\label{T:bound}
Fix $d\geq2$, a domain $\Omega\subset\R^d$ and $p,q\in(1,\infty]$ satisfying \eqref{eq:pq}. Let $a:\Omega\to\R^{d\times d}$ be such that $\lambda$ and $\mu$ given in \eqref{def:lmu} are non-negative and satisfy $\frac1\lambda\in L^q(\Omega)$, $\mu\in L^p(\Omega)$. Then every weak subsolution $u$ of \eqref{eq} in $\Omega$ is locally bounded from above and for every $\gamma>0$ there exists $c=c(\gamma,d,p,q)\in[1,\infty)$ such that for any ball $B_R\subset\Omega$ and $\theta\in(0,1)$ 
\begin{equation}\label{est:T:bound}
 \sup_{B_{\theta R}}u\leq c\frac{\Lambda(B_R)^{\frac{p'}{\gamma}(1+\frac1\delta)}}{(1-\theta)^{\frac{d}{\gamma}s}}\left(\fint_{B_R} u_+^\gamma\right)^\frac1\gamma,
\end{equation}
where $\delta=\min\{\frac{1}{d-1}-\frac1{2p},\frac12\}-\frac1{2q}>0$, $s:=1+p'(1+\frac{1}{\delta})(\frac1p+\frac1q)$, $p'=\frac{p}{p-1}$ and $\Lambda(B_R)$ is defined in \eqref{def:lambda}.
\end{theorem}
\begin{proof}[Proof of part (i) of Theorem~\ref{T:1}]
 Theorem~\ref{T:bound} applied to the subharmonic functions $u$ and $-u$ implies the desired statement.
\end{proof}

As announced in Remark~\ref{rem:bound2d}, we can relax the assumptions $p,q>1$ and $\frac1p+\frac1q<\frac2{d-1}$ of Theorem~\ref{T:1} in the special case $d=2$:
\begin{proposition}\label{P:bound:2d}
Fix a domain $\Omega\subset\R^2$. Let $a:\Omega\to\R^{2\times 2}$ be measurable such that $\lambda$ and $\mu$ given in \eqref{def:lmu} are non-negative and satisfy $\frac1\lambda,\mu\in L^1(\Omega)$. Then there exists $c\in[1,\infty)$ such that for every weak solution $u$ of \eqref{eq} and for any ball $B_R\subset\Omega$
\begin{equation}\label{est:P:bound2d}
 \|u\|_{L^\infty(B_{\frac{R}2})} \leq c\left(R\left(\fint_{B_R}\lambda^{-1}\right)^\frac12\left(\fint_{B_R}a\nabla u\cdot \nabla u\right)^\frac12+\fint_{B_R}|u|\right).
\end{equation}
\end{proposition}

\begin{remark}\label{rem:bound}
A version of Theorem~\ref{T:bound} under the more restrictive integrability condition $\frac1p+\frac1q<\frac2d$ can be found in \cite[Theorem~3.1]{T71}. In view of a counterexample presented in \cite{FSS98}, assumption \eqref{eq:pq} $\frac1p+\frac1q<\frac2{d-1}$ used in Theorem~\ref{T:1} and Theorem~\ref{T:bound} is essentially optimal. Indeed, in \cite[Theorem~2]{FSS98} for every $p,q>1$ satisfying $\frac1p+\frac1q>\frac2{d-1}$ the authors construct a weight $\omega$ with $\omega^{-1}\in L^p(B_1)$ and $\omega\in L^q(B_1)$, and an unbounded weak solution of $-\nabla \cdot \omega\nabla u=0$ in $B_1$ provided $d\geq4$. In fact in \cite[Theorem~2]{FSS98} only the case $d=4$ is considered but the extension to $d\geq5$ is straightforward. In general we cannot say anything about the borderline situation $\frac1p+\frac1q=\frac2{d-1}$, except for the special case $d=2$. 
\end{remark}
Our proof of Theorem~\ref{T:bound} is similar to that of \cite[Theorem~3.1]{T71} and relies on a modification of the Moser iteration method \cite{Moser60,Moser61}. Let us now briefly highlight the main difference of our approach compared to the arguments given in \cite{T71} and discuss from where our improvement comes from. A simple consequence of the H\"older and Sobolev inequality combined with the relation $\frac1p+\frac1q<\frac2d$ is the following weighted Poincar\'e inequality: there exists $\kappa=\kappa(p,q,d)>1$ (in fact $\frac1\kappa=\frac{p}{p-1}(1+\frac1q-\frac2d)$) such that for any ball $B_R$ and $u$ with compact support in $B_R$
\begin{equation}\label{weighted:sobolev}
\left(\fint_{B_R} \mu |u|^{2\kappa}\right)^\frac{1}{\kappa}\leq cR^2\left(\fint_{B_R}\mu^p\right)^\frac1{\kappa p}\left(\fint_{B_R}\lambda^{-q}\right)^\frac1q\fint_{B_R} a\nabla u\cdot\nabla u
\end{equation}
where $c=c(d)\in[1,\infty)$. Inequality \eqref{weighted:sobolev} and the Caccioppoli inequality are enough to use Moser's iteration argument to prove local boundedness. In the situation of Theorem~\ref{T:bound}, i.e.\ with the relaxed assumption \ $\frac1p+\frac1q<\frac2{d-1}$, we do not have a weighted Poincar\'e inequality in the form of \eqref{weighted:sobolev} at hand. However, a version of \eqref{weighted:sobolev} is valid if we replace the $d$-dimensional balls by $(d-1)$-dimensional spheres. In order to exploit this observation, we need an additional optimization step compared to the usual Caccioppoli inequality which is gathered in Lemma~\ref{L:optimcutoff}. The argument for Proposition~\ref{P:bound:2d} is different and in fact much simpler. It is mainly based on the maximum principle and Sobolev inequality in one dimension, see \cite[Proposition~1]{FSS98} for a similar argument.

\begin{proof}[Proof of Theorem~\ref{T:bound}]

Throughout the proof we write $\lesssim$ if $\leq$ holds up to a positive constant that depends only on $d,p$ and $q$.

\step 1 We prove \eqref{est:T:bound} for $\theta=\frac14$, $R=2$ and $\gamma\geq 2p'$, i.e.\ for every $\gamma\geq 2p'$ there exists $c=c(\gamma,d,p,q)\in[1,\infty)$ such that
\begin{equation}\label{est:claim:s1:moser}
 \|u_+\|_{L^\infty(B_{\frac12})}\leq c\Lambda(B_2)^{\frac{p'}{\gamma}(1+\frac1\delta)}\|u_+\|_{L^\gamma(B_2)}.
\end{equation}
For $\beta\geq1$ and $N\in(0,\infty)$, we define
\begin{equation}
 F(u):=F_\beta^N(u)=\begin{cases}(u_+)^\beta &\mbox{for $u\leq N$}\\\beta N^{\beta-1}u-(\beta-1)N^\beta&\mbox{for $u\geq N$}\end{cases}.
\end{equation}
Set $\phi:=\eta^2 F(u)$ with $\eta\geq0$, $\eta\in C_0^1(B_2)$. By \eqref{def:harmonic}, we obtain
\begin{equation}
 \int_\Omega \eta^2 F'(u) a\nabla u\cdot \nabla u\leq -2 \int_\Omega \eta F(u) a\nabla u\cdot \nabla \eta.
\end{equation}
Definition \eqref{def:lmu} (in particular $|a\nabla u|\leq (\mu a\nabla u\cdot \nabla u)^\frac12$ and $a\nabla u\cdot\nabla u\geq\lambda|\nabla u|^2$), Young's inequality and convexity of $F$ in the form of $F(u)\leq u F'(u)$ yield
\begin{equation}\label{est:moser:1}
 \int_\Omega \eta^2 F'(u) \lambda|\nabla u|^2\leq 4 \int_\Omega  F'(u) u^2  \mu|\nabla \eta|^2.
\end{equation}
We rewrite estimate \eqref{est:moser:1} as
\begin{equation}\label{eq:G}
 \int_\Omega \eta^2 \lambda |\nabla G|^2 \leq 4 \int_\Omega  (u G'(u))^2 \mu |\nabla \eta|^2,\qquad\mbox{where}\quad G(u):=\int_0^u |F'(t)|^\frac12 \,dt.
\end{equation}
Fix $\frac12\leq\rho<\sigma\leq2$. We optimize the right-hand side of \eqref{eq:G} with respect to $\eta$ satisfying $\eta\in C_0^1(B_\sigma)$ and $\eta=1$ in $B_\rho$: we use Lemma~\ref{L:optimcutoff}, which by H\"older's inequality implies
\begin{align}\label{est:G1}
 \|\nabla G\|_{L^\frac{2q}{q+1}(B_\rho)}^2\leq \|\lambda^{-1}\|_{L^q(B_\rho)}\int_\Omega \eta^2 \lambda|\nabla G|^2\lesssim  \Lambda(B_\sigma) (\sigma-\rho)^{-\frac{2d}{d-1}}\|u G'(u)\|_{W^{1,p_*}(B_\sigma)}^2,
\end{align}
where $\frac1{p_*}=\min\{\frac12-\frac1{2p}+\frac1{d-1},1\}$ as in Lemma~\ref{L:optimcutoff}. Notice that in order to apply Lemma~\ref{L:optimcutoff}, we used $G'(u)u\in W^{1,p_*}(B_\sigma)$. This is a consequence of $p_*<\frac{2q}{q+1}$ (by \eqref{eq:pq}),  $u\in W^{1,\frac{2q}{q+1}}(\Omega)$ (by H\"older inequality and $u\in H^1(\Omega,a)$), and the chain rule for Sobolev functions, see e.g.\ \cite[Theorem~7.8]{GT}. Set $\alpha=\frac{\beta+1}2$. Then passing with $N\to \infty$ in \eqref{est:G1} (and thus $G(u)=\frac{\sqrt{\beta}}{2\alpha}(u_+^\alpha)$ and $uG'(u)=\sqrt{\beta}(u_+^\alpha)$), we obtain
\begin{align}\label{est:G2}
  \|\nabla (u_+^\alpha)\|_{L^{\frac{2q}{q+1}}(B_\rho)} \lesssim \Lambda(B_\sigma)^\frac12 (\sigma-\rho)^{-\frac{d}{d-1}}\alpha \|u_+^\alpha\|_{W^{1,p_*}(B_\sigma)}.
\end{align}
For future usage, we note that if we choose $\eta\in C_0^1(B_\sigma)$ with $\eta=1$ in $B_\rho$ and $\|\nabla\eta\|_{L^\infty}\leq 2(\sigma-\rho)^{-1}$, estimate \eqref{eq:G}, H\"older inequality and the choice of $G$ as above yield 
\begin{align}\label{est:G3}
  \|\nabla (u_+^\alpha)\|_{L^{\frac{2q}{q+1}}(B_\rho)} \lesssim \Lambda(B_\sigma)^\frac12 (\sigma-\rho)^{-1}\alpha \|(u_+^\alpha)\|_{L^{2p'}(B_\sigma)},
\end{align}
with $p'=\frac{p}{p-1}$. Let us now return to \eqref{est:G2}. Notice that condition \eqref{eq:pq} implies $\frac{2q}{q+1}>p_*$ and thus \eqref{est:G2} contains an improvement in integrability of $\nabla u^\alpha$. 

H\"older's inequality with exponent $\frac{2q}{(q+1){p_*}}$ yields with $\delta=\min\{\frac{1}{d-1} - \frac{1}{2p},\frac12\} - \frac{1}{2q}>0$ (notice \eqref{eq:pq} and $q>1$ imply $\delta>0$) 
\begin{align}\label{est:chi1}
 \biggl( \int_{B_\rho} |\nabla(u_+^{\alpha(1+\delta)})|^{p_*} \biggr)^{\frac 1{p_*}} =& \alpha(1+\delta) \biggl( \int_{B_\rho} |\nabla u_+|^{p_*} u_+^{(\alpha-1){p_*}} u_+^{\alpha\delta {p_*}} \biggr)^{\frac 1{p_*}}\notag\\
  \leq&(1+\delta) \biggl( \int_{B_\rho} |\nabla u_+^\alpha|^\frac{2q}{q+1} \biggr)^{\frac{q+1}{2q}} \biggl( \int_{B_\rho} u_+^\alpha \biggr)^\delta.
\end{align}
Hence, by \eqref{est:G2} with $\chi:=1+\delta>1$,
\begin{align}\label{est:chi2}
  \|\nabla (u_+^{\alpha\chi})\|_{L^{{p_*}}(B_\rho)} \lesssim (\sigma-\rho)^{-\frac{d}{d-1}}\Lambda(B_\sigma)^\frac12\alpha\chi \|u_+^\alpha\|_{W^{1,p_*}(B_\sigma)}^{\chi}.
\end{align}
By Sobolev inequality (using ${p_*} \ge 1$ and $\chi\leq \frac{d}{d-1}$) we get that 
\begin{equation}\label{est:chi3}
 \| u_+^{\alpha\chi} \|_{L^{p_*}(B_\rho)}\lesssim\|u_+^\alpha \|_{W^{1,p_*}(B_\rho)}^{\chi},
\end{equation}
and thus there exists $c=c(d,p,q)\in[1,\infty)$ such that for every $\alpha\geq1$
\begin{align}\label{est:moser:basis}
  \|u_+^{\alpha\chi}\|_{W^{1,p_*}(B_\rho)}^\frac1{\alpha\chi} \leq \left(\frac{c \Lambda(B_\sigma)^\frac12\alpha\chi}{(\sigma-\rho)^{\frac{d}{d-1}}} \right)^\frac1{\alpha\chi}\|u_+^\alpha\|_{W^{1,p_*}(B_\sigma)}^\frac1\alpha.
\end{align}
Estimate \eqref{est:moser:basis} can be iterated in the usual way: Fix $\bar\alpha\geq1$ and for $\nu\in\mathbb N$, set $\alpha_\nu=\bar\alpha \chi^{\nu-1}$, $\rho_\nu=\frac12+\frac1{2^{\nu+1}}$, $\sigma_\nu:=\rho_\nu+\frac1{2^{\nu+1}}=\rho_{\nu-1}$ (where $\rho_0:=1$), and \eqref{est:moser:basis} reads
\begin{align*}
& \|u_+^{\bar \alpha\chi^\nu}\|_{W^{1,p_*}(B_{\rho_\nu})}^\frac1{\bar \alpha\chi^\nu} \leq \left(c \Lambda(B_1)^\frac12\bar \alpha (4\chi)^\nu \right)^\frac1{\bar\alpha \chi^{\nu}}\|u_+^{\bar \alpha \chi^{\nu-1}}\|_{W^{1,p_*}(B_{\rho_{\nu-1}})}^\frac1{\bar\alpha\chi^{\nu-1}}
\end{align*}
and thus
\begin{align}\label{est:moser:almostfinal}
 \|u_+\|_{L^\infty(B_\frac12)}\leq& \prod_{\nu=1}^\infty \left(c\bar \alpha \Lambda(B_1)^\frac12(4\chi)^\nu \right)^{\frac1{\bar \alpha \chi^{\nu}}}\|u_+^{\bar \alpha }\|_{W^{1,p_*}(B_1)}^\frac1{\bar\alpha}\notag\\
 =&\left(c\bar \alpha \Lambda(B_1)^\frac12\right)^{\frac1{\bar \alpha}\sum_{\nu=1}^\infty \frac1{\chi^{\nu}}}\left(4\chi\right)^{\frac1{\bar \alpha}\sum_{\nu=1}^\infty \frac\nu{\chi^{\nu}}}\| u_+^{\bar \alpha}\|_{W^{1,p_*}(B_1)}^{\frac1{\bar \alpha}}.
\end{align}
To estimate the right-hand side of \eqref{est:moser:almostfinal}, we use \eqref{est:G3} and the fact that $p_*<\frac{2q}{q+1}\leq 2\leq 2p'$ 
\begin{align*}
\|\nabla (u_+^{\bar \alpha})\|_{L^{p_*}(B_1)}\lesssim \|\nabla (u_+^{\bar \alpha})\|_{L^{\frac{2q}{q+1}}(B_1)}\lesssim \Lambda(B_2)^\frac12 \bar \alpha\|u_+^{\bar \alpha} \|_{L^{2p'}(B_2)},\qquad \|u_+^{\bar \alpha}\|_{L^{p_*}(B_1)}\lesssim\|u_*^{\bar\alpha}\|_{L^{2p'}(B_2)}.
\end{align*}
Since $\bar \alpha,\Lambda\geq1$ and $\sum_{\nu=1}^\infty \nu \chi^{-\nu}\lesssim1$, we obtain that there exists $c=c(d,p,q,\bar \alpha)\in[1,\infty)$ such that
\begin{equation}
 \|u_+\|_{L^\infty(B_\frac12)}\leq c \Lambda(B_1)^{\frac{1}{2\bar \alpha}(\frac1{1-\chi^{-1}}-1)}\|u_+^{\bar \alpha}\|_{W^{1,p_*}(B_1)}^\frac{1}{\bar \alpha}\lesssim c{\bar \alpha}^\frac1{\bar \alpha}\Lambda(B_2)^{\frac{1}{2\bar \alpha}\frac{\chi}{\chi-1}}\|u_+\|_{L^{2\bar \alpha p'}(B_2)},
\end{equation}
which proves the claim by setting $\gamma=2p'\bar \alpha\geq2p'$ (recall $\chi=1+\delta$).

\step 2 The general case. It is well-known how to lift the result of Step~1 to prove the claim. For convenience of the reader we provide the arguments following the presentation in \cite{HL}. First, by scaling we deduce from \eqref{est:claim:s1:moser} that for $\gamma\geq 2p'$ and $R>0$ 
\begin{equation}\label{est:claim:s1:moser1}
 \|u_+\|_{L^\infty(B_{\frac{R}4})}\leq c\Lambda(B_R)^{\frac{p'}{\gamma}(1+\frac1\delta)}R^{-\frac{d}\gamma}\|u_+\|_{L^\gamma(B_R)},
\end{equation}
where $c=c(\gamma,d,p,q)\in[1,\infty)$ is the same as in \eqref{est:claim:s1:moser}. Now the statement for $\gamma\geq 2p'$ follows by applying for every $y\in B_{\theta R}$ estimate \eqref{est:claim:s1:moser1} with $B_R$ replaced by $B_{(1-\theta)R}(y)$, i.e.
\begin{equation*}
 \|u_+\|_{L^\infty(B_{\frac{1-\theta}4R}(y))}\leq \frac{c\Lambda(B_{(1-\theta)R}(y))^{\frac{p'}{\gamma}(1+\frac1\delta)}}{((1-\theta)R)^{\frac{d}\gamma}}\|u_+\|_{L^\gamma(B_{(1-\theta)R}(y))}\leq \frac{c\Lambda(B_{R})^{\frac{p'}{\gamma}(1+\frac1\delta)}}{(1-\theta)^{\frac{d}\gamma s}R^{\frac{d}\gamma}}\|u_+\|_{L^\gamma(B_R)},
\end{equation*}
where $s=1+p'(1+\frac{1}{\delta})(\frac1p+\frac1q)$ and thus
\begin{equation}\label{est:T:boundbig}
 \|u_+\|_{L^\infty(B_\theta)}\leq \frac{c\Lambda(B_{R})^{\frac{p'}{\gamma}(1+\frac1\delta)}}{(1-\theta)^{\frac{d}\gamma s}R^{\frac{d}\gamma}}\|u_+\|_{L^\gamma(B_R)}\qquad\mbox{for every $\gamma\geq 2p'$.}
\end{equation}

\smallskip

Hence, it remains to prove estimate \eqref{est:T:bound} for $\gamma\in(0,2p')$. For given $\gamma\in(0,2p')$, we first observe that
\begin{equation*}
\| u_+\|_{L^{2p'}(B_R)}\leq \|u_+\|_{L^\infty(B_R)}^{1-\frac{\gamma}{2p'}}\|u_+\|_{L^\gamma(B_R)}^\frac{\gamma}{2p'}
\end{equation*}
and thus by \eqref{est:T:boundbig} (with $\gamma=2p'$) and Youngs inequality
\begin{equation}\label{est:M:afinal}
\begin{split}
\|u_+\|_{L^\infty(B_{\theta R})}\leq& \frac{c\Lambda(B_R)^{{\frac1{2}(1+\frac1\delta)}}}{(1-\theta)^{\frac{d}{2p'}s} R^{\frac{d}{2p'}}}\|u_+\|_{L^{2p'}(B_R)}\leq\frac{c\Lambda(B_R)^{{\frac1{2}(1+\frac1\delta)}}}{(1-\theta)^{\frac{d}{2p'}s} R^{\frac{d}{2p'}}}\|u_+\|_{L^\infty(B_R)}^{1-\frac{\gamma}{2p'}}\|u_+\|_{L^\gamma(B_R)}^\frac{\gamma}{2p'}\\
\leq& \frac12 \|u_+\|_{L^\infty(B_R)}+(2c)^{\frac{2p'}{\gamma}}\frac{\Lambda(B_R)^{{\frac{p'}{\gamma}(1+\frac1\delta)}}}{(1-\theta)^{\frac{d}\gamma s}R^\frac{d}{\gamma}}\|u_+\|_{L^\gamma(B_R)},
\end{split}
\end{equation}
where $c=c(d,p,q)\in[1,\infty)$. Set $f(t):=\|u_+\|_{L^\infty(B_{t})}$, $t\in(0,1]$. The estimate \eqref{est:M:afinal} implies that there exists $c=c(\gamma,d,p,q)\in[1,\infty)$ such that for all $0<r< R\leq 1$
\begin{equation}
f(r)\leq \frac12 f( R)+\frac{c\Lambda(B_R)^{\frac{p'}{\gamma}(1+\frac1\delta)}}{(1-\frac{r}R)^{\frac{d}\gamma s}R^{\frac{d}\gamma}}\|u_+\|_{L^\gamma(B_R)}\leq \frac12 f( R)+\frac{c\Lambda(B_1)^{\frac{p'}{\gamma}(1+\frac1\delta)}}{(R-r)^{\frac{d}\gamma s}}\|u_+\|_{L^\gamma(B_1)}.
\end{equation}
Hence, by \cite[Lemma~4.3]{HL}, we find $c=c(\gamma,d,p,q)\in[1,\infty)$ such that for all $0<r<R<1$,
\begin{equation}
\|u_+\|_{L^\infty(B_{r})}\leq \frac{c\Lambda(B_1)^{\frac{p'}{\gamma}(1+\frac1\delta)}}{(R-r)^{\frac{d}\gamma s}}\|u_+\|_{L^\gamma(B_1)},
\end{equation}
and the claim \eqref{est:T:bound} (with $\theta=r$ and $R=1$) follows.
\end{proof}

\begin{proof}[Proof of Proposition~\ref{P:bound:2d}]

Clearly it suffices to show that every weak subsolution $u$ of \eqref{eq} is locally bounded from above and there exists $c\in[1,\infty)$ such that for any ball $B_R\subset\Omega$ it holds
\begin{equation}\label{est:T:bound:2d}
 \sup_{B_\frac{R}2}u\leq c\left(R\left(\fint_{B_R}\lambda^{-1}\right)^\frac12\left(\fint_{B_R}a\nabla u_+\cdot \nabla u_+\right)^\frac12+\fint_{B_R}u_+\right).
\end{equation}
The maximum principle, see \cite[Theorem~3.1]{T73}, yields
\begin{equation}\label{eq:maxp}
 \sup_{B_R}u\leq \sup_{\partial B_R} u_+\qquad\mbox{for every $B_R\subset \Omega$.}
\end{equation}
In \cite{T73}, the maximum principle \eqref{eq:maxp} is proven for much more general equations. For convenience of the reader we recall the argument for the specific situation considered here at the end of the proof. Next, we prove \eqref{est:T:bound:2d} for $R=1$ (the general case follows by scaling). In view of Fubini's theorem, we find $\rho\in(\frac12,1)$ such that
\begin{equation}\label{eq:slicegood}
 \int_{\partial B_\rho} u_++|\nabla u_+|\leq 2\int_{B_1\setminus B_\frac12} u_++|\nabla u_+|.
\end{equation}
Hence, by the Sobolev inequality in one dimension
\begin{align*}
 \sup_{B_\frac12}u\stackrel{\frac12\leq\rho}{\leq}& \sup_{B_\rho}u\stackrel{\eqref{eq:maxp}}{\leq} \sup_{\partial B_\rho}u_+\lesssim \int_{\partial B_\rho}u_++|\nabla u_+|\stackrel{\eqref{eq:slicegood}}{\leq}2 \int_{B_1} u_++|\nabla u_+|\\
 \leq& 2\|u_+\|_{L^1(B_1)}+2\|\lambda^{-1}\|_{L^1(B_1)}^\frac12\int_{B_1} a\nabla u_+\cdot \nabla u_+,
\end{align*}
where the last inequality follows by H\"older's inequality and \eqref{def:lmu} (in the form $\lambda |\nabla u_+|^2\leq a\nabla u_+\cdot \nabla u_+$).

Finally, we recall the argument of \cite{T73} for \eqref{eq:maxp}. Set $\varphi:=(u-\sup_{\partial B_R} u_+)_+$. Since $\varphi\in H_0^1(B_R,a)$ and $\varphi\geq0$, we can use $\varphi$ as a test function in \eqref{def:harmonic} and obtain
\begin{align*}
0\geq\int_{B_R}a\nabla u\cdot \nabla\varphi\geq\int_{B_R}\lambda|\nabla \varphi|^2\geq \|\lambda^{-1}\|_{L^1(B_R)}^{-1}\|\nabla \varphi\|_{L^1(B_R)}^2,
\end{align*}
and thus, by Poincar\'e inequality and $\varphi=0$ on $\partial B_R$, we obtain $\varphi=0$ and consequently \eqref{eq:maxp}.
\end{proof}

\section{Harnack inequality -- proof of part (ii) of Theorem~\ref{T:1} and some applications}\label{sec:harnack}

The main result of this section is the following
\begin{theorem}[Weak Harnack inequality]\label{T:weakharnack}
Fix $d\geq2$, a domain $\Omega\subset\R^d$ and $p,q\in(1,\infty]$ satisfying \eqref{eq:pq}. Let $a:\Omega\to\R^{d\times d}$ be such that $\lambda$ and $\mu$ given in \eqref{def:lmu} are non-negative and satisfy $\frac1\lambda\in L^q(\Omega)$, $\mu\in L^p(\Omega)$. Let $u$ be a non-negative weak supersolution of \eqref{eq} in $\Omega$. Then, for every $0<\theta<\tau<1$, $\gamma\in(0,\frac{q_*}2)$, where $\frac1{q_*}=\frac12+\frac1{2q}-\frac1d$ if $1+\frac1q>\frac2d$ and $q_*=+\infty$ otherwise (i.e.\ if $d=2$ and $q=\infty$), and any $B_R\subset \Omega$  there exists $C=C(d,p,q,\theta,\tau,\gamma,\Lambda(B_R))\in[1,\infty)$ such that
\begin{equation}\label{eq:weakharnack}
 \left(\frac{1}{R^d}\int_{B_{\tau R}} u^\gamma\right)^\frac1\gamma\leq C\inf_{B_{\theta R}} u.
\end{equation}
In fact, the constant $C$ in \eqref{eq:weakharnack} satisfies $C\leq c_1 e^{\Lambda(B_R)c_2}$ with $c_1=c_1(\gamma,d,p,q,\tau,\theta)\in[1,\infty)$ and $c_2=c_2(\gamma,d,p,q)>0$.
\end{theorem}
\begin{proof}[Proof of part (ii) of Theorem~\ref{T:1}]
 Notice $q_*>1$ for every $d\geq2$. Hence, combining the local boundedness estimate \eqref{est:T:bound} with $\gamma=\frac12$ and Theorem~\ref{T:weakharnack} ($q_*>1$ allows $\gamma=\frac12$), we obtain
\begin{equation*}
\sup_{B_\frac{R}2} u\stackrel{\eqref{est:T:bound}}{\leq} c\Lambda(B_R)^{2p'(1+\frac1\delta)}\left(\frac1{R^d}\int_{B_{\frac34 R}}u^\frac12\right)^2\stackrel{\eqref{eq:weakharnack}}{\leq}  c\Lambda(B_R)^{2p'(1+\frac1\delta)}C\inf_{B_\frac{R}2}u,
\end{equation*} 
with $c=c(d,p,q)\in[1,\infty)$ and $C=C(d,p,q,\Lambda(B_R))\in[1,\infty)$ which proves the claim.
\end{proof}

In \cite{T71}, Trudinger proved the conclusion of Theorem~\ref{T:weakharnack} under the more restrictive assumption $\frac1p+\frac1q<\frac2d$. We prove Theorem~\ref{T:weakharnack} by combining the strategy of Trudinger in the proof of \cite[Theorem 4.1]{T71} with the local boundedness result Theorem~\ref{T:bound} and an improved Caccioppoli inequality due to Lemma~\ref{L:optimcutoff}. Even though experts might already anticipate how to adapt the arguments of \cite{T71}, we give a detailed proof at the end of this section. 

\smallskip

Before proving Theorem~\ref{T:weakharnack}, we list some consequences of Theorem~\ref{T:1} which are by now standard and thus we only give the statements without proofs. In the uniformly elliptic setting, Harnack inequality implies H\"older continuity of  weak solutions to \eqref{eq}. As observed in \cite{T71}, due to the explicit dependence of the constant $c$ in \eqref{est:harnack:strong} on $\Lambda(B_R)$ this is in general not true anymore in the non-uniformly elliptic setting. However, Theorem~\ref{T:1} yields the following large-scale H\"older continuity:

\begin{corollary}[H\"older continuity 'on large scales']\label{C:holder}
Consider the situation of Theorem~\ref{T:1}. For $R>0$ set $\bar \Lambda_R:=\sup_{R'\geq R}\Lambda(B_R)$. Suppose that $u$ is weak solution of \eqref{eq} in $B_{R_0}\subset\Omega$ and $\bar \Lambda_{R_1}<\infty$ for some $0<R_1<\frac14 R_0$. Then, for all $R\in[R_1,\frac12 R_0]$
\begin{equation}
\underset{B_{ R}}{\osc}\,u\leq c\left(\frac{R}{R_0}\right)^\theta \left(\fint_{B_{R_0}} |u| \right),
\end{equation}
where $c$ and $\theta$ are positive constants depending on $d,p,q$ and $\bar \Lambda_{R_1}$.
\end{corollary}

%

\begin{remark}
 If $\Lambda(B)$ is bounded uniformly for all balls $B\subset \Omega$, Corollary~\ref{C:holder} implies 'usual' H\"older continuity of $a$-harmonic functions in $\Omega$. This improves \cite[Theorem~5.1]{T71} since it relaxes the integrability assumptions on $\lambda^{-1}$ and $\mu$. However, under such local assumptions H\"older regularity of $a$-harmonic functions is proven under much weaker integrability conditions, see for example \cite{FKS82}.
\end{remark}
A direct consequence of Corollary~\ref{C:holder} is the following zero-order Liouville property:
\begin{corollary}[Liouville Theorem]\label{C:liouville}
Fix $d\geq2$ and $p,q\in(1,\infty]$ satisfying \eqref{eq:pq}. Let $a:\R^d\to\R^{d\times d}$ be a measurable coefficient field such that $\lambda,\mu$ given by \eqref{def:lmu} are non-negative and satisfy $\frac1\lambda\in L^q_{\rm loc}(\R^d)$, $\mu\in L^p_{\rm loc}(\R^d)$. Moreover, suppose that $\limsup_{R\to\infty}\Lambda(B_R)<\infty$. Then, every bounded local weak solution $u$ of \eqref{eq} in $\R^d$, in the sense of Definition~\ref{def:solution}, is constant.
\end{corollary}
In \cite[Theorem~3]{EP73}, the conclusion of Corollary~\ref{C:liouville} is proven (relying on the results of Trudinger in \cite{T71,T73}) under the more restrictive assumption $\frac1p+\frac1q<\frac2d$ (notice that \cite[Theorem~3]{EP73} applies also in situations with additional lower order terms which are not considered here).

\smallskip

Finally, we provide the proof of Theorem~\ref{T:weakharnack}.

\begin{proof}[Proof of Theorem~\ref{T:weakharnack}]

Throughout the proof we write $\lesssim$ if $\leq$ holds up to a positive constant which depends only on $d,p$ and $q$. Without loss of generality we set $R=1$ and suppose that $u\geq \e>0$. In what follows, we suppose $1+\frac1q>\frac2d$. The remaining case $d=2$ with $q=\infty$ can be done with no additional difficulties by appealing to corresponding versions of Sobolev inequality.

\step 1 Fix $0<\theta<\tau<1$. We claim that
\begin{equation}\label{harnack:claim:s1}
 \exp\left(\fint_{B_\tau}\log (u)\right)\leq C\inf_{B_{\theta}} u,
\end{equation}
where $C=c_1 e^{\Lambda(B_1)^{c_2}}$ with $c_1=c_1(d,p,q,\tau,\theta)\in[1,\infty)$ and $c_2=c_2(d,p,q)\in[1,\infty)$. Testing \eqref{def:harmonic} with $\phi=\eta u^{-1}$ with $\eta\geq0$ and $\eta\in H_0^1(B_1,a)$, we obtain
\begin{equation}\label{eq:harnack:1}
\int_{B_1} \frac1u a\nabla u\cdot \nabla \eta\,dx-\int_{B_1} \frac\eta{u^{2}} a\nabla u \cdot \nabla u \,dx\geq0.
\end{equation} 
Setting $v:=\log(\frac{k}u)$, $k>0$, we see $a\nabla v\cdot\nabla \eta=-\frac1u a\nabla u\cdot \nabla \eta$, hence \eqref{eq:harnack:1} and $u>0$ imply 
\begin{equation*}
\int_{B_1} a\nabla v\cdot \nabla \eta \,dx\leq0.
\end{equation*} 
The arbitrariness of $\eta$ implies that $v$ is $a$-subharmonic in the sense of \eqref{def:harmonic}. Hence Theorem~\ref{T:bound} with $\gamma=q_*$, where $\frac{1}{q_*}=\frac12+\frac1{2q}-\frac1d$ (recall the assumption $1+\frac1q>\frac2d$), yield
\begin{equation}\label{s1:harnack:1}
 \sup_{B_\theta}v\lesssim \Lambda(B_\tau)^{\frac{p'}{q_*}(1+\frac1\delta)}(1-\tfrac\theta\tau)^{-\frac{d}{q_*}s}\left(\fint_{B_\tau} |v|^{q_*}\right)^\frac1{q_*}\leq \Lambda(B_1)^{\frac{p'}{q_*}(1+\frac1\delta)}(\tau-\theta)^{-\frac{d}{q_*}s}\left(\int_{B_\tau} |v|^{q_*}\right)^\frac1{q_*},
\end{equation}
with $s=1+p'(1+\frac{1}{\delta})(\frac1p+\frac1q)$. Next, we replace $\eta$ in \eqref{eq:harnack:1} by $\eta^2$ with $\eta\in C_0^1(B_1)$ and obtain (using \eqref{def:lmu} and applying Young's inequality)
\begin{equation*}
\int_{B_1}\eta^2 \lambda|\nabla v|^2\leq 4\int_{B_1}\mu|\nabla \eta|^2. 
\end{equation*}
Choosing $\eta$ such that $\eta=1$ in $B_\tau$ and $|\nabla \eta|\leq 2(1-\tau)^{-1}$, we obtain
\begin{equation}\label{s1:harnack:2}
\int_{B_\tau}\lambda |\nabla v|^2\leq 4^2(1-\tau)^{-2}\int_{B_1}\mu. 
\end{equation}
Finally, we choose $k>0$ such that $\fint_{B_\tau}v=0$, i.e.\ $k:=\exp\left(\fint_{B_\tau} \log(u)\right)$, and thus by a combination of H\"older and Sobolev inequality (note that $q_*$ is the Sobolev exponent for $\frac{2q}{q+1}$)
\begin{equation}\label{s1:harnack:3}
\left(\int_{B_\tau} |v|^{q_*}\right)^\frac1{q_*}\lesssim \left(\int_{B_\tau} |\nabla v|^{\frac{2q}{q+1}}\right)^\frac{q+1}{2q}\leq \|\lambda^{-1}\|_{L^q(B_\tau)}^\frac12 \left(\int_{B_\tau} \lambda |\nabla v|^2\right)^\frac{1}{2}.
\end{equation}
Combining \eqref{s1:harnack:1}--\eqref{s1:harnack:3}, we obtain
\begin{align*}
 \sup_{B_\theta}v\stackrel{\eqref{s1:harnack:1}}{\lesssim}& \frac{\Lambda(B_1)^{\frac{p'}{q_*}(1+\frac1\delta)}}{(\tau-\theta)^{\frac{d}{q_*}s}}\left(\int_{B_\tau} |v|^{q_*}\right)^\frac1{q_*} \stackrel{\eqref{s1:harnack:3}}{\lesssim} \frac{\Lambda(B_1)^{\frac{p'}{q_*}(1+\frac1\delta)}}{(\tau-\theta)^{\frac{d}{q_*}s}}\|\lambda^{-1}\|_{L^q(B_\tau)}^\frac12 \left(\int_{B_\tau} \lambda |\nabla v|^2\right)^\frac{1}{2}\\
 \stackrel{\eqref{s1:harnack:2}}{\lesssim}&\frac{\Lambda(B_1)^{\frac{p'}{q_*}(1+\frac1\delta)+\frac12}}{(\tau-\theta)^{\frac{d}{q_*}s}(1-\tau)}.
\end{align*}
Finally, the definitions of $v$ and $k$ yield the claimed estimate \eqref{harnack:claim:s1}.

\step 2 Fix $0<\theta<\tau<1$. We claim that there exist $s_0=s_0(d,p,q,\theta,\tau,\Lambda(B_1))>0$ and $C=C(d,p,q,\theta,\tau,\Lambda(B_1))<\infty$ such that 
\begin{equation}\label{claim:harnack:s2}
 \left(\int_{B_{\theta }} |u|^{s_0}\right)^\frac1{s_0}\leq C\exp\left(\fint_{B_{\tau }}\log (u)\right).
\end{equation}
In fact, it holds $s_0^{-1},C\leq c_1e^{\Lambda(B_1)^{c_2}}$ with $c_1=c_1(d,p,q,\tau,\theta)\in[1,\infty)$ and $c_2=c_2(d,p,q)$. Set $w:=v_-=(\log(\frac{u}k))_+$.  For given  $\beta\geq1$ and $\eta\geq0,\, \in C_0^1(B_1)$, we consider the test function
\begin{equation}
 \phi(x)=\eta^2(x)u^{-1}(x)(w^{\beta}(x)+(2\beta)^\beta).
\end{equation}
The fact that $u$ is $a$-superharmonic, see~\eqref{def:harmonic}, $w\geq0$, and the elementary inequality (coming from Young's inequality)
\begin{equation}\label{ineq:elementary}
 \beta w^{\beta-1}\leq \frac12 (w^\beta + (2\beta)^\beta)
\end{equation}
yield (using $\nabla \phi =\frac{2\eta}{u} (w^\beta + (2\beta)^\beta)\nabla\eta+\frac{\eta^2}{u^2}(\beta w^{\beta-1}-w^\beta-(2\beta)^\beta)\nabla u$)
\begin{equation}\label{est:harnack:s2a}
  \frac12\int \frac{\eta^2}{u^{2}} (\beta w^{\beta-1} + \frac12(2\beta)^\beta) a\nabla u \cdot \nabla u\leq\frac12 \int \frac{\eta^2}{u^{2}} (w^\beta + (2\beta)^\beta) a\nabla u \cdot \nabla u\leq 2\int \frac{\eta}{u} (w^\beta + (2\beta)^\beta) a \nabla u \cdot \nabla \eta.
\end{equation}
Appealing to \eqref{def:lmu} and Young's inequality, we estimate the right-hand side of \eqref{est:harnack:s2a}
\begin{align}\label{est:harnack:s2a2}
2\int \frac\eta{u} (w^\beta + (2\beta)^\beta) a \nabla u \cdot \nabla \eta \leq \frac14\int \frac{\eta^2}{u^2} (\beta w^{\beta-1}+(2\beta)^\beta) a \nabla u\cdot \nabla u+4\int (\frac1\beta w^{\beta+1}+(2\beta)^\beta) \mu |\nabla\eta|^2.
\end{align}
Note that the first term on the right-hand side in \eqref{est:harnack:s2a2} can be absorbed into the left-hand side of \eqref{est:harnack:s2a} and we obtain, using $\beta\geq1$ and the definition of $w$,
\begin{equation}\label{est:harnack:s2b}
  \beta \int \eta^2 w^{\beta-1}  a\nabla w \cdot \nabla w\leq 16\int  (w^{\beta+1} + (2\beta)^\beta) \mu |\nabla \eta|^2.
\end{equation}
Fix $0<\theta\leq \rho<\sigma\leq \tau<1$. Let $\eta\in C_0^1(B_\sigma)$ be such that $\eta=1$ in $B_\rho$. Minimizing the right-hand side of \eqref{est:harnack:s2b} among such cutoff functions, we obtain with the help of Lemma~\ref{L:optimcutoff}, \eqref{def:lmu} and $\beta\geq1$
\begin{equation}
  \frac{2}{\beta+1} \int_{B_\rho}  \lambda|\nabla w^{\frac{\beta+1}2}|^2\leq c(\sigma-\rho)^{-\frac{2d}{d-1}}\|\mu\|_{L^p(B_\sigma)}\left(\|w^\frac{\beta+1}2\|_{W^{1,p_*}(B_\sigma)}^2+(2\beta)^\beta\right),
\end{equation}
where $c=c(d,p,\theta)\in[1,\infty)$. Hence, setting $\alpha:=\frac{\beta+1}2$, we obtain by H\"older inequality and \eqref{def:lambda} 
\begin{equation}\label{est:moser:harnack:s2:1}
  \|\nabla (w^\alpha)\|_{L^\frac{2q}{q+1}(B_\rho)} \leq \frac{c(\alpha \Lambda(B_1))^\frac12}{(\sigma-\rho)^{\frac{d}{d-1}}}\left(\|w^\alpha\|_{W^{1,p_*}(B_\sigma)}+(4\alpha)^\alpha\right).
\end{equation}
Using \eqref{est:chi1}--\eqref{est:chi3} (with $u$ replaced by $w$), we derive from \eqref{est:moser:harnack:s2:1} an analogue of \eqref{est:moser:basis}
\begin{align}\label{est:moser:basisharnack}
  \|w^{\alpha\chi}\|_{W^{{1,p_*}}(B_\rho)}^\frac1{\alpha\chi} \leq \left(\frac{c \Lambda(B_1) \alpha\chi}{(\sigma-\rho)^{\frac{2d}{d-1}}} \right)^\frac1{2\alpha\chi}\left(\|w^\alpha\|_{W^{1,p_*}(B_\sigma)}^\frac1\alpha+4\alpha\right),
\end{align}
where $\chi=1+\delta$, with $\delta=\min\{\frac1{d-1}-\frac1{2p},\frac12\}-\frac1{2q}>0$, and $c=c(d,p,q,\theta,\tau)\in[1,\infty)$. For $\nu\in\mathbb N$ set, $\alpha_\nu=\chi^{\nu-1}$, $\rho_\nu=\theta+2^{-\nu}(\tau-\theta)$ and $\sigma_\nu:=\rho_\nu+2^{-\nu}(\tau-\theta)=\rho_{\nu-1}$. Then
\begin{align}\label{harnack:est:iterate}
&\| w^{\chi^\nu}\|_{W^{1,p_*}(B_{\rho_{\nu}})}^{\frac{1}{\chi^\nu }}\leq \left(\frac{c \Lambda(B_1) (4^{\frac{d}{d-1}}\chi)^{\nu}}{(\tau-\theta)^{\frac{2d}{d-1}}}\right)^\frac1{2\chi^\nu}\left(  \|w^{\chi^{\nu-1}}\|_{W^{1,p_*}(B_{\rho_{\nu-1}})}^{\frac{1}{\chi^{\nu-1} }}+4\chi^{\nu-1}\right).
\end{align}
Estimate \eqref{harnack:est:iterate} can be iterated and we find that there exists $c_1=c_1(d,p,q,\theta,\tau)\in[1,\infty)$ and $c_2=c_2(d,p,q)\in[1,\infty)$ such that for every $s\geq1$,
\begin{equation}\label{est:harnack:s2:ws}
 \|w\|_{L^s(B_\theta)}\leq c_1 \Lambda(B_1)^{c_2} \left(\|w\|_{W^{1,p_*}(B_{\tau})} + s\right).
\end{equation}
Recalling the fact $w=v_-$, estimates \eqref{s1:harnack:2} and \eqref{s1:harnack:3} and the fact $p_*\leq \frac{2q}{q+1}$ (by \eqref{eq:pq}), we obtain for $s\geq1$
\begin{equation}\label{est:harnack:s2:ws2}
 \|w\|_{L^s(B_\theta)}\leq c_1' \Lambda(B_1)^{c_2'} s,
\end{equation}
where $c_1'$ and $c_2'$ have the same dependencies as $c_1$ and $c_2$ in \eqref{est:harnack:s2:ws}, respectively. Estimate \eqref{est:harnack:s2:ws2} and the choice $s_0:=(2c_1'\Lambda(B_1)^{c_2'}e)^{-1}$ yield for every $j\in\mathbb N$
\begin{equation*}
 \frac{s_0^j\|w\|_{L^j(B_\theta)}^j}{j!}\leq \frac1{2^j}\frac{j^j}{e^jj!}\leq \frac1{2^j},
\end{equation*}
and thus
\begin{equation*}
 \int_{B_\theta}\exp(s_0 w)\leq \sum_{j=0}^\infty \frac{s_0 \|w\|_{L^j(B_\theta)}^j}{j!}\leq \sum_{j=0}^\infty \frac1{2^j}=2.
\end{equation*}
Recall that $w=(\log \frac{u}{k})_+$, with $k=\exp\left(\fint_{B_\tau}\log(u)\right)$, and thus
\begin{equation*}
 \left(\int_{B_\theta} (\frac{u}k)^{s_0}\right)^\frac1{s_0}\leq (2+|B_1|)^\frac1{s_0}\qquad \Rightarrow\qquad \left(\int_{B_\theta} u^{s_0}\right)^\frac1{s_0}\leq (2+|B_1|)^{2ec_1'\Lambda(B_1)^{c_2'}}\exp\left(\fint_{B_\tau} \log(u)\right),
\end{equation*}
which proves the claim.

\step 3 Fix $\e\in(0,\frac{q_*}2)$ and $0<\tau<\tau'<1$. We claim that for every $\gamma\in(\e,\frac{q_*}2)$ there exists $C=C(\gamma,d,\e,p,q,\Lambda(B_1),\tau,\tau')\in[1,\infty)$ such that
\begin{equation}\label{claim:s3:harnack}
\left(\int_{B_\tau} u^\gamma\right)^\frac1\gamma\leq C\left(\int_{B_{\tau'}} u^\e\right)^\frac1\e,
\end{equation}
where $C\leq c_1\Lambda^{c_2\frac{\gamma}\e}$ with $c_1=c_1(\gamma,d,p,q,\tau,\tau')\in[1,\infty)$ and $c_2=c_2(\gamma,d,p,q)\in[1,\infty)$. Recall $u>0$. Testing \eqref{def:harmonic} with $\phi:=\eta^2 u^\beta$ where $\beta\in(-1,0)$ and $\eta\geq0$, $\eta\in C_0^1(B_1)$, we obtain
\begin{equation}
 \int \eta^2 u^{\beta-1} a\nabla u\cdot \nabla u\leq \frac{2}{|\beta|} \int \eta u^\beta a\nabla u\cdot \nabla \eta.
\end{equation}
Young's inequality and \eqref{def:lmu} (in the form $|a\nabla u|\leq (\mu a\nabla u\cdot\nabla u)^\frac12$ and $a\nabla u\cdot\nabla u\geq \lambda |\nabla u|^2$) yield
\begin{equation}\label{est:moser:1:harnack}
 \int \eta^2 \lambda |\nabla (u^\frac{\beta+1}2)|^2\leq \frac{(\beta+1)^2}{|\beta|^2} \int  u^{\beta+1}  \mu|\nabla \eta|^2.
\end{equation}
Set $\alpha=\frac{\beta+1}2\in(0,\frac12)$. Fix $\tau\leq\rho<\sigma\leq\tau'$ and consider $\eta\in C_0^1(B_\sigma)$ satisfying $\eta=1$ in $B_\rho$. We estimate the right-hand side of \eqref{est:moser:1:harnack} by either applying Lemma~\ref{L:optimcutoff} or choosing a linear cutoff function $\eta$. In combination with H\"older inequality, we get $c=c(p,d,\tau)\in[1,\infty)$ such that
\begin{align}
  \|\nabla (u^\alpha)\|_{L^{\frac{2q}{q+1}}(B_\rho)} \leq& c\Lambda^\frac12 (\sigma-\rho)^{-\frac{d}{d-1}}\frac{\alpha}{1-2\alpha} \|u^\alpha\|_{W^{1,p_*}(B_\sigma)}\label{est:apriori:harnack:s3a}\\
   \|\nabla (u^\alpha)\|_{L^{\frac{2q}{q+1}}(B_\rho)} \leq& c\Lambda^\frac12 (\sigma-\rho)^{-1}\frac{\alpha}{1-2\alpha} \|u^\alpha\|_{L^{2p'}(B_\sigma)}.\label{est:apriori:harnack:s3b}
\end{align}
Using \eqref{est:chi1}--\eqref{est:chi3}, we derive from \eqref{est:apriori:harnack:s3a}
\begin{align}\label{est:moser:basis:harnack}
  \|u^{\alpha\chi}\|_{W^{1,p_*}(B_\rho)}^\frac1{\alpha \chi} \leq \left(\frac{c \Lambda^\frac12\chi\alpha}{(\sigma-\rho)^{\frac{d}{d-1}}(1-2\alpha)} \right)^\frac1{\alpha\chi}\|u^\alpha\|_{W^{1,p_*}(B_\sigma)}^\frac1\alpha,
\end{align}
where $\chi=1+\delta$, with $\delta=\min\{\frac{1}{d-1}-\frac1{2p},\frac12\}-\frac1{2q}>0$. Fix $\kappa\in(0,\frac12)$. We set $\alpha_0:=\frac{\kappa}{\chi}\in(0,\frac12)$ and $\alpha_i=\frac{\kappa}{\chi^{i+1}}=\frac{\alpha_{i-1}}{\chi}$ for $i\in \mathbb N$. Fix $n\in\mathbb N$ such that $\alpha_n<\e\frac{p-1}{2p}\leq \alpha_{n-1}$, $\tau_1:=\tau+\frac14(\tau'-\tau)$, and $\tau_2:=\tau+\frac34(\tau'-\tau)$. Iterating \eqref{est:moser:basis:harnack} $n$-times, we find $C=C(\kappa,d,\e,p,q,\tau,\tau',\Lambda(B_1))\in[1,\infty)$ such that
\begin{equation}\label{est:s3:harnack:inter}
\|u^{\kappa}\|_{W^{1,p_*}(B_{\tau_1})}^\frac1{\kappa}\leq C \|u^{\alpha_n}\|_{W^{1,p_*}(B_{\tau_2})}^\frac1{\alpha_n}.
\end{equation} 
Observe that the assumption $\e\frac{p-1}{2p}\leq \alpha_{n-1}$ yields $\chi^n\leq \frac{\kappa}\e \frac{2p}{p-1}$ and thus $C\leq c_1\Lambda^{\frac{\kappa}\e c_2}$ with $c_1=c_1(\kappa,d,\e,p,q,\tau,\tau')\in[1,\infty)$ and $c_2=c_2(d,p,q)\in[1,\infty)$. Using $p_*<\frac{2q}{q+1}$, \eqref{est:apriori:harnack:s3b} and the choice of $\alpha_n$ (i.e.\ $\alpha_n<\e\frac{p-1}{2p}$), we estimate the right-hand side of \eqref{est:s3:harnack:inter}
\begin{align}\label{est:s3:harnack:right}
 \|u^{\alpha_n}\|_{W^{1,p_*}(B_{\tau_2})}^\frac1{\alpha_n}\leq C\|u\|_{L^{\frac{2p}{p-1}\alpha_n}(B_{\tau'})}\leq C\|u\|_{L^{\e}(B_{\tau'})}.
\end{align}
Using Sobolev inequality and \eqref{est:apriori:harnack:s3a}, we get for the left-hand side of \eqref{est:s3:harnack:inter} 
\begin{equation}\label{est:s3:harnack:left}
\|u^\kappa\|_{L^{ q_*}(B_\tau)}\lesssim \|\nabla (u^\kappa)\|_{L^\frac{2q}{q+1}(B_\tau)}+|B_\tau|^{\frac1{q*}-\frac1{p_*}}\|u^\kappa\|_{L^{p_*}(B_\tau)}\leq  c \|u^\kappa\|_{W^{1,p_*}(B_{\tau_1})},
\end{equation}
where $\frac1{q*}=\frac12+\frac1{2q}-\frac1d$. Then a combination of \eqref{est:s3:harnack:inter}--\eqref{est:s3:harnack:left} yields the desired claim \eqref{claim:s3:harnack} for $\gamma=\kappa q_*\in(0,\frac12 q_*)$. 

\step 4 Conclusion. Fix $0<\theta<\tau<1$. Combining Step~1, 2 and 3 (with $\e=s_0$), we obtain
\begin{align*}
 \left(\int_{B_\tau} u^\gamma\right)^\frac1\gamma\stackrel{\eqref{claim:s3:harnack}}{\leq} C_1\left(\int_{B_{\frac{\tau+1}2}} u^{s_0}\right)^\frac1{s_0}\stackrel{\eqref{claim:harnack:s2}}{\leq} C_1C_2\exp\left(\fint_{B_{\frac{\tau+3}4}}\log(u)\right)\stackrel{\eqref{harnack:claim:s1}}{\leq} C_1C_2C_3\inf_{B_\theta} u,
\end{align*}
where $C_1,C_2,C_3\in[1,\infty)$ satisfy the desired dependencies. 
\end{proof}

\section{Sublinear corrector in random homogenization with degenerate coefficients}\label{sec:prob}

In this section we apply Theorem~\ref{T:1} in the context of stochastic homogenization. Stochastic homogenization for uniformly elliptic equations dates back to the classical papers \cite{PV79,K79}. Currently, stochastic homogenization beyond uniform ellipticity is an active field of research, see e.g.\ \cite{ADS15,AN17,AD18,AS14,BFO18,CD16,DNS,FK97,FHS17,NSS17}.

A central object in the homogenization of linear elliptic equations is the so called corrector: For $\xi\in \R^d$, the corrector $\phi_\xi$ is characterized almost surely by solving
\begin{equation}\label{eq:cor0}
\nabla\cdot a^\omega(\nabla \phi_\xi+\xi)=0\qquad \mbox{in $\R^d$}\quad\mbox{and}\qquad \lim_{R\to\infty}\frac1R\fint_{B_R}|\phi_\xi|=0.
\end{equation}
Here, we assume that the coefficient fields $\{a^\omega(x)\}_{x\in\R^d}\subset \R^{d\times d}$ are \textit{statistically homogeneous and ergodic}, and \textit{non-uniformly elliptic} (see below for the precise assumptions). In \cite{ADS15,CD16,DNS,FK97}, the corrector $\phi$ is used prominently to prove quenched invariance principles for the random walk \cite{ADS15,DNS} or diffusion \cite{CD16,FK97} in a random environment with degenerate and/or unbounded coefficients.  The key ingredient in \cite{ADS15,CD16,DNS,FK97} is to upgrade the $L^1$-sublinearity into $L^\infty$-sublinearity, i.e.\ to show $\frac1R\|\phi_\xi\|_{L^\infty(B_R)}\to0$ as $R\to\infty$. In this section, we show that the results of Section~\ref{sec:bound} can be used to weaken the assumption of \cite{CD16,FK97} in order to establish $L^\infty$-sublinearity of the corrector. In order to reduce input from probability theory in the present paper, we postpone the application to the quenched invariance principle for the random walk to a forthcoming work.

\smallskip

Let us now be more precise and phrase the assumptions on the coefficient fields by appealing to the language of ergodic, measure preserving dynamical systems (which is a standard in the theory of stochastic homogenization; see, e.g., the seminal paper \cite{PV79}): Let $(\Omega,\mathcal F,\mathbb P)$ denote a probability space and $\tau=(\tau_x)_{x\in\R^d}$ a family of measurable mappings $\tau_x:\Omega\to\Omega$ satisfying
\begin{itemize}
  \item (group property) $\tau_0\omega=\omega$ for all $\omega\in\Omega$ and $\tau_{x+y}=\tau_x\tau_y$ for all $x,y\in\R^d$.
  \item (stationarity) For every $x\in\R^d$ and $B\in\mathcal F$ it holds $\mathbb P(\tau_x B)=\mathbb{P}(B)$.
  \item (ergodicity) All $B\in\mathcal F$ with $\tau_x B=B$ for all $x\in\R^d$ satisfy $\mathbb P(B)\in\{0,1\}$.
\end{itemize}
For a random field $a:\Omega\to\R^{d\times d}$ and $\omega\in\Omega$, we denote by $a^\omega:\R^d\to\R^{d\times d}$ its stationary extension given by $a^\omega(x):=a(\tau_x\omega)$.
\begin{assumption}\label{ass:prop}
There exists exponents $p,q\in[1,\infty]$ satisfying $\frac1p+\frac1q<\frac2{d-1}$ if $d\geq3$ such that the following is true: The random variables $\lambda,\mu$ given by
  \begin{equation}\label{def:lmurandom}
 \lambda(\omega):=\inf_{\xi\in\R^d} \frac{\xi\cdot a(\omega)\xi}{|\xi|^2},\qquad \mu(\omega):=\sup_{\xi\in\R^d}\frac{|a(\omega)\xi|^2}{\xi\cdot a(\omega)\xi}
 \end{equation}
 are non-negative and satisfy the moment condition
 \begin{equation}\label{ass:moment}
 \mathbb E[\lambda^{-q}]<\infty,\qquad \mathbb E[\mu^p]<\infty,
 \end{equation}
 where $\mathbb E$ denotes the expected value.
\end{assumption}
Assumption~\ref{ass:prop} ensures the existence of a well-defined corrector which is the subject of the following
\begin{lemma}\label{L:corrector:prelim}
Suppose that Assumption~\ref{ass:prop} is satisfied. Then there exists $\Omega_1\subset \Omega$ with $\mathbb P(\Omega_1)=1$ such that the following is true: For every $\omega\in\Omega_1$ and all $i=1,\dots,d$ there exists a weak solution $\phi_i\in H_{\rm loc}^1(\R^d,a^\omega)$ of 
\begin{equation}\label{eq:corrector}
\nabla\cdot a^\omega (e_i+\nabla \phi_i)=0\qquad\mbox{in $\R^d$}
\end{equation}
with $a^\omega(x)=a(\tau_x \omega)$, which is sublinear in $L^1$ in the sense
\begin{equation}\label{sublinearl1}
 \limsup_{R\to\infty}R^{-1}\left(\fint_{B_R}|\phi_i|\right)=0
\end{equation}
and for every $z\in\mathbb Q^d$ it holds
\begin{equation}\label{ene:ergodic}
 \limsup_{R\to\infty}\left(\fint_{B_R(Rz)}a(\nabla \phi_i+e_i)\cdot (\nabla \phi_i+e_i)\right)\leq\mathbb E[\mu].
\end{equation}
\end{lemma}
We omit the proof of Lemma~\ref{L:corrector:prelim} since it is by now standard. In fact, if $a$ is supposed to be symmetric a stronger statement can be found in \cite[Section 4]{CD16}. Appealing to an additional truncation argument as e.g.\ in \cite{AM91,FK97} similar arguments as in \cite[Section 4]{CD16} can be used to cover also the non-symmetric case.

\smallskip

Now we state the main result of this section, namely the almost sure $L^\infty$-sublinearity of the corrector

\begin{proposition}\label{P:sublinear}
Suppose that Assumption~\ref{ass:prop} is satisfied. Then there exists $\Omega_2\subset \Omega_1$ with $\mathbb P(\Omega_2)=1$ such that the following is true: For every $\omega\in\Omega_2$ and all $i=1,\dots,d$ the functions $\phi_i \in H_{\rm loc}^1(\R^d,a^\omega)$ satisfying \eqref{eq:corrector} and \eqref{sublinearl1} are sublinear in $L^\infty$ in the sense 
\begin{equation}\label{sublinerlinfty}
 \limsup_{R\to\infty}R^{-1}\|\phi_i\|_{L^\infty(B_R)}=0.
\end{equation}
\end{proposition}

\begin{remark}
In \cite{CD16}, the sublinearity of the corrector in the form \eqref{sublinerlinfty} is shown under moment conditions \eqref{ass:moment} with the more restrictive relation $\frac1p+\frac1q<\frac2d$ (see also \cite{ADS15} for a similar result in the discrete setting and \cite{FK97} for a corresponding statements with strictly elliptic, unbounded coefficients). In two dimensions, Proposition~\ref{P:sublinear} might be not surprising since an analogous statement in a discrete setting was already proven by Biskup in \cite{Biskup}. For completeness and since the argument in \cite{Biskup} uses the discrete structure we include this case here. The counterexample to local boundedness in \cite{FSS98} suggests that Assumption~\ref{ass:prop} should be almost optimal for the conclusion of Proposition~\ref{P:sublinear}. In fact it is recently shown by Biskup and Kumagai \cite{BK14} in a discrete setting, that the corresponding statement of Proposition~\ref{P:sublinear} fails if \eqref{ass:moment} only holds for $p,q$ satisfying $\frac1p+\frac1q>\frac{2}{d-1}$ provided $d\geq4$.
\end{remark}

\begin{proof}[Proof of Proposition~\ref{P:sublinear}]

Throughout the proof we write $\lesssim$ if $\leq$ holds up to a positive constant which depends only on $d,p$ and $q$. Before we give the details of the proof, we briefly explain the idea. There are two obstructions to deduce the statement directly from Theorem~\ref{T:1}: Firstly, we are not able to prove local boundedness of the corrector by considering \eqref{eq:corrector} as an equation for $\phi_i$ with the right-hand side $\nabla \cdot ae_i$ as it is e.g.\ done in \cite{CD16}. Secondly and more severe, the right-hand side $\nabla \cdot ae_i$ is not small in general. We overcome this issues by appealing to a two-scale argument: We introduce an additional length-scale $\rho R$ with $0<\rho\ll1$ and compare $\phi_i$ on balls with radius $\sim \rho R$ with $a$-harmonic functions $\phi_i+e_i\cdot x-c$ with a suitable chosen $c\in \R$. Using the $L^1$-sublinearity of $\phi_i$ and the fact that the linear part coming from $e_i\cdot x$ can be controlled by $\rho>0$ on each ball of radius $\sim \rho R$ we obtain the desired claim. 

\step 1 As a preliminarily step, we recall the needed input from ergodic theory. In view of the spatial ergodic theorem, we obtain from the moment condition \eqref{ass:moment} that there exists $\Omega'\subset \Omega$ with $\mathbb P(\Omega')=1$  such that for $\omega\in\Omega'$ it holds $\lambda(\tau_{(\cdot)}\omega)^{-1}\in L_{\rm loc}^q(\R^d)$, $\mu(\tau_{(\cdot)}\omega)\in L_{\rm loc}^p(\R^d)$, and for every $z\in\mathbb Q^d$
\begin{equation}\label{eq:ergodic}
 \lim_{R\to\infty}\fint_{B_R(Rz)}\lambda(\tau_x \omega)^{-q}\,dx=\mathbb E[\lambda^{-q}] \quad\mbox{and}\quad \lim_{R\to\infty}\fint_{B_R(Rz)}\mu(\tau_x \omega)^{p}\,dx=\mathbb E[\mu^{p}],
\end{equation}
see e.g.\ \cite[Theorem~11.18]{KLO}.

\step 2 Conclusion for $p,q>1$. We set $\Omega_2:=\Omega_1\cap \Omega'$, where $\Omega'$ is given as in Step~1. Clearly $\Omega_2$ has full measure. From now on we fix $\omega\in\Omega_2$. 

\smallskip

Fix $\rho\in(0,\frac{R}2]$  and cover $B_R$ with finitely many balls $B_{d \rho R }(Rz)$, $z\in \rho \mathbb Z^d\cap B_1=:\mathcal Z_{\rho}$. For $z\in \mathcal Z_{\rho}$, set $u_i^z(x):=\phi_i(x)+e_i\cdot(x-z)$. Obviously, \eqref{eq:corrector} implies that $u_i^z$ is $a^\omega$-harmonic. Hence, \eqref{est:C:bound} (with $\gamma=1$) yields
\begin{align}\label{moser:ui}
\|u_i^z\|_{L^\infty(B_{d \rho R }(Rz))}\lesssim& \Lambda(B_{2d \rho R }(Rz))^{p'(1+\frac1\delta)}\fint_{B_{2d \rho R }(Rz)}|u_i^z|\,dx\notag\\
\lesssim&\Lambda(B_{2d \rho R }(Rz))^{p'(1+\frac1\delta)}\left(\fint_{B_{2d \rho R }(Rz)}|\phi_i(x)|\,dx+\rho R\right),
\end{align}
where $p=\frac{p}{p-1}$, $\delta=\min\{\frac1{d-1}-\frac1{2p},\frac12\}-\frac1{2q}>0$. Estimate \eqref{moser:ui} implies the following $L^\infty$ estimate on $\phi_i$
\begin{align}\label{est:phiexi:1}
 \|\phi_i\|_{L^\infty(B_R)}\leq \sup_{z\in \mathcal Z_{\rho}}\|\phi_i\|_{L^\infty(B_{d\rho R}(Rz))}&\leq  \sup_{z\in \mathcal Z_{\rho}}\|u_i\|_{L^\infty(B_{d\rho R}(Rz))}+d\rho R\notag\\
 &\lesssim \sup_{z\in \mathcal Z_{\rho}}\Lambda(B_{2d\rho R}(Rz)))^{p'(1+\frac1\delta)}\left(\fint_{B_{2d\rho R}(Rz))}|\phi_i|\,dx+\rho R\right)+\rho R\notag\\
 &\lesssim (\rho^{-d}\fint_{B_{2d R}}|\phi_i(x)|\,dx+\rho R)\sup_{z\in \mathcal Z_{\rho}}\Lambda(B_{2d\rho R}(Rz)))^{p'(1+\frac1\delta)}+\rho R.
\end{align}
Clearly for every $z\in \mathcal Z_\rho$, it holds $\Lambda(B_{2d\rho R}(Rz))=\Lambda(B_{R'}(R'z'))$ with $R'=2d\rho R$ and $z'=(2d\rho)^{-1}z\in (2d)^{-1}\mathbb Z^d\cap B_{(2d\rho)^{-1}}=:\mathcal Z'_\rho$. Using \eqref{eq:ergodic} and the fact that $\mathcal Z_{\rho}$ (and thus $\mathcal Z_\rho'$) is a finite set, we obtain
\begin{align}\label{limsupLambdae}
\limsup_{R\to\infty}\sup_{z\in \mathcal Z_{\rho}}\Lambda(B_{2d\rho R}(Rz))=&\limsup_{R\to\infty}\sup_{z'\in \mathcal Z'_{\rho}}\left(\fint_{B_R(Rz')}\lambda(\tau_x \omega)^{-q}\,dx\right)^\frac1q\left(\fint_{B_R(Rz')}\mu(\tau_x \omega)^{p}\,dx\right)^\frac1p\notag\\
\leq& \mathbb E[\lambda^{-q}]^\frac1q\mathbb E[\mu^p]^\frac1p.
\end{align}
Finally, we combine \eqref{est:phiexi:1} and \eqref{limsupLambdae} with the $L^1$-sublinearity of $\phi_i$, i.e.\ \eqref{sublinearl1}, to obtain
\begin{equation*}
 \limsup_{R\to\infty}R^{-1}\|\phi_i\|_{L^\infty(B_R)}\lesssim\rho(\mathbb E[\lambda^{-q}]^\frac1q\mathbb E[\mu^p]^\frac1p)^{p'(1+\frac1\delta)}+\rho.
\end{equation*}
The arbitrariness of $\rho>0$ implies \eqref{sublinerlinfty} and finishes the proof.

\step 3 The remaining case: $d=2$ and $p=q=1$. Let $\Omega_2$ be as in Step~2. From now on we fix $\omega\in\Omega_2$ and use the same notation as in Step~2.   

\smallskip

Using estimate \eqref{est:P:bound2d} instead of \eqref{est:C:bound}, we obtain 
\begin{align}\label{moser:ui:2d}
&\|u_i^z\|_{L^\infty(B_{2 \rho R }(Rz))}\notag\\
\lesssim&\fint_{B_{4 \rho R }(Rz)}|\phi_i(x)|\,dx+\rho R+\rho R\left(\fint_{B_{4 \rho R }(Rz)}\lambda^{-1}\right)^\frac12\left(\fint_{B_{4 \rho R }(Rz)}a(\nabla \phi_i+e_i)\cdot (\nabla \phi_i+e_i)\,dx\right)^\frac12,
\end{align}
instead of \eqref{moser:ui}. Estimate \eqref{moser:ui:2d} implies the following $L^\infty$ estimate on $\phi_i$

\begin{align}\label{est:phiexi:1:2d}
 &\|\phi_i\|_{L^\infty(B_R)}\leq \sup_{z\in \mathcal Z_{\rho}}\|\phi_i\|_{L^\infty(B_{2\rho R}(Rz))}\leq  \sup_{z\in \mathcal Z_{\rho}}\|u_i\|_{L^\infty(B_{2\rho R}(Rz))}+2\rho R\notag\\
 \lesssim& \sup_{z\in \mathcal Z_{\rho}}\left(\fint_{B_{4 \rho R }(Rz)}|\phi_i|+\rho R+\rho R\left(\fint_{B_{4 \rho R }(Rz)}\lambda^{-1}\right)^\frac12\left(\fint_{B_{4 \rho R }(Rz)}a(\nabla \phi_i+e_i)\cdot (\nabla \phi_i+e_i)\right)^\frac12\right)\notag\\
 \lesssim& \rho^{-d}\fint_{B_{2d R}}|\phi_i|+\rho R+\rho R\sup_{z\in \mathcal Z_{\rho}}\left(\fint_{B_{4 \rho R }(Rz)}\lambda^{-1}\right)^\frac12\left(\fint_{B_{4 \rho R }(Rz)}a(\nabla \phi_i+e_i)\cdot (\nabla \phi_i+e_i)\right)^\frac12.
\end{align}
Using \eqref{ene:ergodic} and \eqref{eq:ergodic}, we obtain similar to \eqref{limsupLambdae}
\begin{align}\label{limsupLambdae2d}
\limsup_{R\to\infty}\sup_{z\in \mathcal Z_{\rho}}\left(\fint_{B_{4 \rho R }(Rz)}\lambda^{-1}\right)^\frac12\left(\fint_{B_{4 \rho R }(Rz)}a(\nabla \phi_i+e_i)\cdot (\nabla \phi_i+e_i)\right)^\frac12
\leq& \mathbb E[\lambda^{-1}]^\frac12\mathbb E[\mu]^\frac12.
\end{align}
Finally, we combine \eqref{est:phiexi:1:2d} and \eqref{limsupLambdae2d} with the $L^1$-sublinearity of $\phi_i$, i.e.\ \eqref{sublinearl1}, to obtain
\begin{equation*}
 \limsup_{R\to\infty}R^{-1}\|\phi_i\|_{L^\infty(B_R)}\lesssim\rho\mathbb E[\lambda^{-1}]^\frac12\mathbb E[\mu]^\frac12+\rho.
\end{equation*}
The arbitrariness of $\rho>0$ implies \eqref{sublinerlinfty} and finishes the proof.

\end{proof}


\begin{thebibliography}{99}
\bibitem{ADS15} S.\ Andres, J.-D.\ Deuschel, and M.\ Slowik, Invariance principle for the random conductance model in a degenerate ergodic environment, \textit{Ann.\ Probab.,} {\bf 43}, 1866--1891 (2015).
\bibitem{AN17} S.\ Andres and S.\ Neukamm, Berry-Esseen Theorem and Quantitative homogenization for the Random Conductance Model with degenerate Conductances, \textit{Stoch PDE: Anal Comp} (2018). https://doi.org/10.1007/s40072-018-0127-8.
\bibitem{AD18} S.\ N.\ Armstrong and P.\ Dario, Elliptic regularity and quantitative homogenization on percolation clusters, \textit{Comm. Pure Appl. Math.,} {\bf 71}, 1717--1849 (2018). 
\bibitem{AS14} S.\ N.\ Armstrong and C.\ K. Smart.\ Regularity and stochastic homogenization of fully nonlinear equations without uniform ellipticity, \textit{Ann.\ Probab.,} {\bf 42}, 2558--2594 (2014).
\bibitem{AM91} M.\ Avellaneda and A.\ J.\ Majda, An integral representation and bounds on the effective diffusivity in passive advection by laminar and turbulent flows, \textit{Comm. Math. Phys.,} {\bf 138}, 339--391 (1991).
\bibitem{BFO18} P.\ Bella, B.\ Fehrman and F.\ Otto, A Liouville theorem for elliptic systems with degenerate ergodic coefficients, \textit{Ann. Appl. Probab.,} {\bf 28}, 1379--1422 (2018).
\bibitem{Biskup} M.\ Biskup, Recent progress on the random conductance model, \textit{Probab.\ Surv.} \textbf 8, 294--373 (2011). 
\bibitem{BK14} M.\ Biskup and T.\ Kumagai, Quenched Invariance Principle for a class of random conductance models with long-range jumps, \textit{ arXiv:1412.0175}
\bibitem{CW86} S.\ Chanillo and R.\ L.\ Wheeden, Harnack's inequality and mean-value inequalities for solutions of degenerate elliptic equations, \textit{Comm. Partial Differential Equations,} {\bf 11} (10), 1111--1134 (1986).
\bibitem{CD16} A.\ Chiarini and J.-D.\ Deuschel, Invariance principle for symmetric diffusions in a degenerate and unbounded stationary and ergodic random medium, \textit{Ann. Inst. Henri Poincaré Probab. Stat.,} {\bf 52}, 1535--1563 (2016).
\bibitem{CMM18} G.\ Cupini, P.\ Marcellini and E.\ Mascolo, Nonuniformly elliptic energy integrals with $p,q$-growth, \textit{Nonlinear Anal.,} {\bf 177}, 312--324 (2018).
\bibitem{DG57} E.\ De Giorgi, Sulla differenziabilità e l'analiticità delle estremali degli integrali multipli regolari, \textit{Mem. Accad. Sci. Torino. Cl. Sci. Fis. Mat. Nat.,} (3) {\bf 3}, 25--43 (1957).
\bibitem{DNS} J.-D.\ Deuschel, T.\ A.\ Nguyen and M.\ Slowik, Quenched invariance principles for the random conductance model on a random graph with degenerate ergodic weights, \textit{Probab. Theory Related Fields,} {\bf 170}, 363--386  (2018).
\bibitem{EP73} D.\ E.\ Edmunds and L.\ A.\ Peletier, A Liouville theorem for degenerate elliptic equations, \textit{J.\ London Math.\ Soc.}, {\bf 7}, 95--100 (1973). 
\bibitem{FKS82} E.\ Fabes, C.\ Kenig and R.\ Serapioni, The local regularity of solutions to degenerate elliptic equations, \textit{Comm.\ Partial Differential Equations,} {\bf 7},  77--116 (1982).
\bibitem{FSS98} B.\ Franchi, R.\ Serapioni and F.\ Serra Cassano, Irregular solutions of linear degenerate elliptic equations, \textit{Potential Anal.} {\bf 9} (1998), no. 3, 201--216.
\bibitem{FK97} A.\ Fannjiang and T.\ Komorowski, A martingale approach to homogenization of unbounded random flows, \textit{Ann. Probab.,} {\bf 25},  1872--1894 (1997).
\bibitem{FHS17} F.\ Flegel, M.\ Heida, and M.\ Slowik, Homogenization theory for the random conductance model with degenerate ergodic weights and unbounded-range jumps, Preprint, available at arXiv:1702.02860 (2017)
\bibitem{GT} D.\ Gilbarg and N.\ Trudinger, Elliptic partial differential equations of second order, Springer, 1998. 
\bibitem{HL} Q.\ Han and F.\ Lin, \textit{Elliptic partial differential equations}, Courant Lecture Notes in Mathematics, vol. 1, New York University, Courant Institute of Mathematical Sciences, New York; American Mathematical Society, Providence, RI, 1997.
\bibitem{KLO} T.\ Komorowski, C.\ Landim and S.\ Olla, \textit{Fluctuations in Markov Processes: Time Symmetry and Martingale Approximation}, Grundlehren Math. Wiss. 345, Springer, Heidelberg, 2012.
\bibitem{K79} S.\ M.\ Kozlov, The averaging of random operators, \textit{Mat. Sb. (N.S.),} {\bf 109} (1979) 188--202, 327.
\bibitem{Moser60} J.\ Moser, A new proof of De Giorgi's theorem concerning the regularity problem for elliptic differential equations, \textit{Comm.\ Pure Appl.\ Math.,} {\bf 13}, 457--468 (1960).
\bibitem{Moser61} J.\ Moser, On Harnack's theorem for elliptic differential equations, \textit{Comm.\ Pure Appl.\ Math.,} {\bf 14}, 577--591 (1961).
\bibitem{MS68} M.\ K.\ V.\ Murthy and G.\ Stampacchia, Boundary value problems for some degenerate-elliptic operators, \textit{Ann. Mat. Pura Appl.} (4) {\bf 80}, 1--122 (1968). 
\bibitem{Nash58} J.\ Nash, Continuity of solutions of parabolic and elliptic equations, \textit{Am.\ J.\ Math.,}  {\bf 80}, 931--954 (1958).
\bibitem{NSS17} S.\ Neukamm, M.\ Sch\"affner, and A.\ Schl\"omerkemper, Stochastic homogenization of nonconvex discrete energies with degenerate growth, \textit{SIAM J.\ Math.\ Anal.,} {\bf 49}, 1761--1809 (2017).
\bibitem{PV79} G. C. Papanicolaou and S. R. S. Varadhan, Boundary value problems with rapidly oscillating random coefficients, in \textit{Random Fields, Colloq. Math. Soc. János Bolyai,} 27, 835--873 (1981).
\bibitem{T71} N.\ Trudinger, On the regularity of generalized solutions of linear, non-uniformly elliptic equations, \textit{Arch.\ Rational Mech.\ Anal.}, {\bf 42}, 50--62 (1971).
\bibitem{T73} N.\ Trudinger, Linear elliptic operators with measurable coefficients, \textit{Ann.\ Scuola Norm.\ Sup.\ Pisa,} {\bf 27}, 265--308 (1973).
\end{thebibliography}
\end{document}